\documentclass[11pt]{article}
\usepackage[utf8]{inputenc}
\usepackage[T1]{fontenc}
\usepackage[english]{babel}
\usepackage{graphicx}
\usepackage{amssymb, amsmath, amsthm, mathrsfs, parskip, amsthm, enumerate, mathtools, subfig, comment}
\usepackage[toc,page]{appendix}
\usepackage[margin=0.3 in]{geometry}
\newgeometry{includefoot,left=2.5cm,right=2.5cm, bottom=3cm,top=3cm}
\usepackage{comment}

\newcommand{\R}{\mathbb{R}}

\newcommand{\dist}{\mathrm{dist}}

\providecommand{\Linf}{\mathcal{L}_{\infty}}
\providecommand{\Lpos}{\mathcal{L}_{\infty}^+}
\providecommand{\Lneg}{\mathcal{L}_{\infty}^-}
\providecommand{\Rn}{\mathbb{R}^n}
\providecommand{\Rnn}{\mathbb{R}^{n+1}}
\providecommand{\Ns}{N \times I}
\providecommand{\Rnp}{D_{\Omega}}

\providecommand{\On}{\Omega_E}
\providecommand{\Oe}{\Omega_{E}^{t_1,t_2}}

\providecommand{\cyl}{\mathcal{Q}_{T}}
\providecommand{\RnpO}{D_{\Omega}\setminus \Omega}

\providecommand{\ug}{\mathcal{U}_g}
\providecommand{\ups}{\overline{H}_g}

\providecommand{\lps}{\underline{H}_g}
\providecommand{\lcg}{\mathcal{L}_g}

\newtheorem{thm}{Theorem}[section]
\newtheorem{lem}[thm]{Lemma}
\newtheorem{prop}[thm]{Proposition}
\newtheorem{cor}[thm]{Corollary}
\newtheorem{rmk}[thm]{Remark}
\newtheorem{exm}[thm]{Example}

\numberwithin{equation}{section}

\theoremstyle{definition}
\newtheorem{definition}[thm]{Definition}

\title{On an evolutionary equation involving the nonlocal infinity Laplacian}

\author{Frida Fejne}
\date{}

\begin{document}

\maketitle
\thispagestyle{empty}

\begin{abstract}
\noindent We study the evolutionary equation
\begin{align*} 
 \frac{\partial u}{\partial t} = (\mathcal{L}_{\infty})^{\alpha}u   \textup{ in }  \Omega,
\end{align*}
where $(\mathcal{L}_{\infty})^{\alpha}u$ denotes a nonlocal infinity Laplacian acting on the function $u$, $0<\alpha \leq 1$ and $\Omega$ is a bounded open set in $\mathbb{R}^{n+1}$. We prove existence and uniqueness using Perron's method for the Dirichlet problem when $\Omega$ is a cylinder and $0<\alpha<1$.
\end{abstract}
\section{Introduction}
In this paper, we study Perron's method for the evolutionary problem  
\begin{align} \label{eq:Luf} 
 \frac{\partial u}{\partial t} = (\mathcal{L}_{\infty})^{\alpha}u   \textup{ in }  \Omega,
\end{align}
where $0<\alpha \leq 1$ and
\begin{equation} \label{eq:infinitylap}
(\mathcal{L}_{\infty})^{\alpha}u(x,t) = \sup_{y \in \Rn} \frac{u(y,t)-u(x,t)}{|x-y|^{\alpha}} + \inf_{y \in \Rn} \frac{u(y,t)-u(x,t)}{|x-y|^{\alpha}},
\end{equation}
 is a nonlocal version of the infinity Laplace operator. This amounts to considering the corresponding Dirichlet problem
\begin{align} \label{eq:Luf2} 
\begin{cases}
 \frac{\partial u}{\partial t} = (\mathcal{L}_{\infty})^{ \alpha}u &  \textup{ in }  \Omega, \\
 u = g &  \textup{ on }  \Omega^c,
\end{cases}
\end{align}
for a bounded and continuous boundary function $g$. Here, $\Omega$ denotes a bounded open set in $\mathbb{R}^{n+1}$.

We show the existence of viscosity solutions to \eqref{eq:Luf} for $0 < \alpha <1$, by constructing the upper and lower Perron solutions of \eqref{eq:Luf2} and relating the boundary behavior to the existence of barriers. That is, we prove that when there exist barriers for all boundary points, except possibly for a few specific points $(x_0,t_0) \in \partial \Omega$, then the upper and lower Perron solutions coincide and is a viscosity solution to \eqref{eq:Luf2}.  This is the content of the main theorem of this paper, see Theorem \ref{thm:existence}. The uniqueness follows from the comparison principle, which we also prove. In the case where $\Omega$ is a cylinder, we construct explicit barriers for the boundary points on the ``bottom'' and the ``side'' of the cylinder. This is done for $0<\alpha \leq 1$. 

The paper is organized as follows. In Section 2 we introduce some notation, basic definitions, and fundamental results. In Section \ref{sec:Perron} we construct viscosity solutions to \eqref{eq:Luf} using Perron's method. This is followed by Section \ref{sec:barriers}, in which we study barriers for the boundary value problem \eqref{eq:Luf2} and show the existence of solutions in the case where $\Omega$ is a cylinder, and $0<\alpha<1$.
\subsection{Background and related work}
The operator in \eqref{eq:infinitylap} is closely related to the non-local fractional $p$-Laplace operator
\begin{equation*}\label{eq:fraclap}
(-\mathcal{L}_p)^{\alpha} u(x) = \lim_{\epsilon \to 0} \int_{\Rn \setminus B_{\epsilon}(x)} 2\frac{|u(y)-u(x)|^{p-2}(u(y)-u(x))}{|x-y|^{\alpha p}} \ dy,
\end{equation*}
where $\alpha \in (0,1]$. If we for instance consider
\begin{align*}
     \begin{cases}
                (-\mathcal{L}_p)^{\alpha} u  = 0 & \textup{ in } \Omega, \\
                    u =g & \textup{ on } \Rn \setminus \Omega, 
     \end{cases}
\end{align*}
for $0 < \alpha \leq 1$, and let $p \to \infty$, the limiting problem becomes
\begin{align}\label{eq:Holderinf}
     \begin{cases}
                 (\Linf)^{\alpha} u = 0 & \textup{ in } \Omega, \\
                    u =g & \textup{ on } \Rn \setminus \Omega. 
     \end{cases}
\end{align}
For more on the limiting equations in different settings when sending $p$ to infinity, see \cite{CLM2012, FL2016, LL2012}. We would like to emphasize that there are no standard definitions of the nonlocal fractional $p$-Laplace operator and the nonlocal infinity Laplace operator, see for example \cite{BCF20122, CJ2017,  STD2018} for some different definitions of the fractional $p$-Laplace operator, and \cite{BCF2012} for another definition of the nonlocal infinity Laplace operator. For more related work on the these operators, we refer the reader to \cite{ CRY2010, CS2007, SR2019, L2006, L2016, PSSW2011} and the references therein. In \cite{JK2006} Juutinen and Kawohl studied basic properties of the parabolic counter part of the infinity Laplace equation. Among other things they established the existence of viscosity solutions through approximating equations.
     For other evolutionary problems involving different versions of the infinity Laplacian, see, e.g., \cite{AJK2009,CW2003}. In \cite{HL2016}, Hynd and Lindgren investigated properties of weak and viscosity solutions of the nonlocal equation
     \begin{equation*}
         | v_t|^{p-2}v_t +(-\Delta_p)^s v=0,
     \end{equation*}
     where $(-\Delta_p)^s$ denotes the fractional $p$-Laplace operator. For more on the fractional heat equation and nonlocal evolutionary problems, see e.g. \cite{BPSV2014,   BSV2017, LN2023, NS2016, V2018}. In \cite{TEEV2023} del Teso et al. studied the evolutionary initial value problem
    \begin{align*}
        \begin{cases}
            \partial_t u= \Delta^s_{\infty} u& \textup{for  } x \in \Rn, t >0,\\
        u(x,0) =u_0(x)& \textup{for  } x \in \Rn,
    \end{cases}
\end{align*} 
where 
\begin{equation*}
\Delta^s_{\infty} \phi := C_s \sup_{|y|=1} \inf_{|\tilde{y}|=1}\int_0^{\infty} (\phi(x + \eta y) + \phi(x - \eta \tilde{y}) - 2 \phi(x)) \frac{d \eta}{\eta^{1+2s}},
\end{equation*}
with $s \in (1/2,1)$ and $n\geq 2$. They developed an existence theory for viscosity solutions and showed a comparison theorem between classical and viscosity solutions. Note, however, that this version of the nonlocal infinity Laplace operator is quite different from the one studied in this paper.

When it comes to proving the existence of viscosity solutions to a partial differential equation that is not in divergence form, it is common to consider either approximating sequences and the vanishing viscosity method or Perron's method, which relies on the existence of a comparison principle. With Perron's method, the aim is to construct a solution that attains some prescribed boundary function $g$. In general, this is not possible. However, one may often construct the upper and the lower Perron solutions that both attain the boundary values in some sense. It is of great interest to investigate whether the upper and lower Perron solutions coincide, and under what conditions they actually do attain the right boundary values. The method was developed by Perron in 1923 in order to solve the Dirichlet boundary value problem for Laplace's equation \cite{P1923}. In 1929, Sternberg observed that the method could be extended to the heat equation \cite{S1929}. Perron's method has also been studied in the degenerate and singular setting of the $p$-parabolic equation, see for instance \cite{BBUP2015,KL1996} and it
can also be used for equations involving integro-differential operators, and other nonlocal operators, as long as it is possible to prove a suitable comparison principle. For some references of comparison principles for viscosity solutions, see  \cite{ BI2008, CS2009, CIL1992}, and for Perron's method for nonlocal problems, see for example \cite{BCI2008, bbk2024perron,   I1987, I1989, K2004, K2005, KKP2017, LL2017,  M2017}.

\subsection*{Acknowledgments}
The author would like to thank Erik Lindgren for suggesting this problem as well as for many interesting and productive discussions. The research was supported by VR Grant 2016-03639.

\section{Notation, definitions and basic properties}
This section contains some general definitions and basic results. In Section \ref{subsec:generalnotation} we define the notation that we will use throughout the paper. This section is followed by Section \ref{sec:viscositysol} where we define viscosity solutions to the equation in \eqref{eq:Luf}. In Section \ref{sec:comparisonprinciple} we prove the comparison principle for viscosity sub- and supersolutions to \eqref{eq:Luf2} for $0 <\alpha<1$.

\subsection{General notation}\label{subsec:generalnotation}
Throughout this paper, $\Omega$ will denote a bounded open set in $\Rn \times \R$, and $\Rnp=\Rn \times [T_0,T]$, where
\begin{equation*}
    T_0 = \inf\{t \in \R: (x,t) \in \Omega\}, \quad T = \sup\{t \in \R: (x,t) \in \Omega\}.
\end{equation*}
By $\mathcal{B}_r(x_0,t_0)$ we denote the open ball with radius $r$ that is centered at the point $(x_0,t_0) \in \mathbb{R}^{n+1}$. We use $B_r(x_0,t_0)$ to denote the $n$ dimensional open ball with radius $r$, centered at $(x_0,t_0)$, that is 
$$
B_r(x_0,t_0) = \{(x,t_0) \in \Rn: |x-x_0|<r \}.
$$
%Furthermore, for $\xi_0 = (x_0,t_0)$ we define the half ball $B^-_r(\xi_0)$ as
%\begin{equation}\label{eq:halfball}
 %   B^-_r(\xi_0)= \{(x,t) \in B_r(\xi_0):t < t_0 \}.
%\end{equation}
We define the distance between a point $x \in \Rn$ and the set $E \subset \Rn$ as
\begin{equation*}
    d(x, E) = \inf_{\tilde{x} \in E} |x-\tilde{x}|.
\end{equation*}
For an open and bounded set $E \subset \Rn$ we define the cylinder $\cyl$ as
\begin{equation}\label{eq:cyl}
    \cyl = E \times (T_0,T),
\end{equation}
where $T_0,T \in \R$, and we define 
\begin{equation}\label{eq:Nr}
   Q_r(x,t)=B_r(x,t) \times (t-r,t+r).
\end{equation}
When it is clear from context what we mean, we sometimes omit the point $(x,t)$ and simply write $Q_r$. 

We define the projection of $\Omega \in \Rnn$ onto $\Rn$ as
\begin{equation}\label{eq:projx}
    \On = \{x \in \Rn:(x,t) \in \Omega \textup{ for some } t \in \R \}.
\end{equation}
We are going to denote the the upper semicontinuous envelope of a function $u$ as
\begin{equation}\label{eq:u_upp}
    u^*(x,t) := \lim_{\epsilon \to 0} \sup_{(y,s) \in B_{\epsilon}(x,t)} u(y,s),
\end{equation}
and the lower semicontinuous envelope of a function $v$ as
\begin{equation}\label{eq:u_low}
    v_*(x,t) := \lim_{\epsilon \to 0} \inf_{(y,s) \in B_{\epsilon}(x,t)} v(y,s).
\end{equation}
When we work with the operator in \eqref{eq:infinitylap} we decompose it as
$$
(\Linf)^{\alpha} u = (\Lpos)^\alpha u + (\Lneg)^{\alpha} u,
$$
where
\begin{align*}
(\Lpos)^\alpha u(x,t) = \sup_{y \in \Rn, y \neq x} \frac{u(y,t)-u(x,t)}{|x-y|^{\alpha}},
\end{align*}
and
\begin{align*}
(\Lneg)^\alpha u(x,t) = \inf_{y \in \Rn, y \neq x} \frac{u(y,t)-u(x,t)}{|x-y|^{\alpha}}.
\end{align*}
 We usually omit the supscript $\alpha$ and sometimes we omit the subscript $y \neq x$.

%Furthermore, for a point $x_0 \in \Rn$ we define the set of points $(x_0,t)$ that belong to $\Omega$ for some $t$ as
%\begin{equation}\label{eq:projt_x0}
 %   \Omega_t^{x_0} = \{(x_0,t) \in \mathbb{
%R}^{n+1}: (x_0,t) \in \Omega \ \textup{ for some }t \in \R\}.
%\end{equation}
\begin{comment}
Given two points $t_1$ and $t_2$ we denote the set of points in $\Omega$ for which $t_1 < t < t_2$ as
\begin{equation}\label{eq:projPartomega}
    \Omega(t_1,t_2)  = \{(x,t) \in \mathbb{R}^{n+1} :(x,t) \in \Omega \textup{ for some }  t_1<t <t_2\}.
\end{equation}
The projection of $\Omega(t_1,t_2)$ onto the $x$ axis is defined as
\begin{align}\label{eq:projxt}
     \Oe = \{x \in \Rn :(x,t) \in \Omega(t_1,t_2) \textup{ for some } t \in \R\}. 
\end{align}
\end{comment}

\subsection{Viscosity solutions}\label{sec:viscositysol}
\begin{definition}\label{def:viscsol}
    Let $K \in \R$. A function  $u: \Rnp \to \R$ is a viscosity subsolution of 
    \begin{equation}\label{eq:viscsol}
        \frac{\partial u}{\partial t}-\Linf u = K,
    \end{equation}
    in $\Omega$ if the following conditions hold:
    \begin{itemize}
        \item $u$ is upper semicontinuous in $\Omega$ and satisfies
    \begin{equation}\label{eq:growthcond}
          |u(x,t)| \leq C(1+|x|^{\beta}),  
    \end{equation}
         for all $x \in \Rn$, $t\in [T_0,T]$, and where $C> 0$, $\beta < \alpha$. 
        \item Let $(x_0,t_0) \in \Omega$ and let $\Ns \subset \Omega $ be a neighborhood of $(x_0,t_0)$. If $\varphi \in C^1(\Ns)$, such that $\varphi(x_0,t_0)=u(x_0,t_0)$ and $\varphi \geq u$ in $\Ns$, then
    $$
    \frac{\partial \varphi}{\partial t} \left(x_0,t_0\right)  - \mathcal{L}_{\infty}\varphi_{N}(x_0,t_0) \leq K, 
    $$
    where
\begin{align}\label{eq:varphiN}
     \varphi_{N}(x,t)= \begin{cases}
                      \varphi & \textup{if  } \ (x,t) \in \Ns,\\
                      u & \textup{if  } (x,t) \in \ \Rnp   \setminus \Ns.
            \end{cases}
 \end{align} 
\end{itemize}
We say that $u$ is a viscosity supersolution in $\Omega$ if $-u$ is a viscosity subsolution in $\Omega$. Finally, $u$ is a viscosity solution in $\Omega$ if it is both a viscosity subsolution and a viscosity supersolution.
\end{definition}
\begin{rmk}\label{rmk:posneg}
    We note that due to the growth condition \eqref{eq:growthcond} it follows that $\Lneg u \leq 0$ and $\Lpos u \geq 0$. 
\end{rmk}

Sometimes $N= B_r(x_0,t_0)$ and $I= (t_0-r, t_0+r)$. In that case, we will use the notation
\begin{align}\label{eq:defvarphir}
     \varphi_r(x,t)= \begin{cases}
                      \varphi & \textup{if  } \ (x,t) \in Q_r(x_0,t_0),\\
                      u & \textup{if  } (x,t) \in \ \Rnp   \setminus  Q_r(x_0,t_0),
            \end{cases}
 \end{align}
where $Q_r(x,t)$ is defined in \eqref{eq:Nr}.

\begin{rmk}\label{rmk:strict_touching}
    We note that in Definition \ref{def:viscsol} we may assume that the touching is strict. Because if $u$ is a subsolution and $\varphi \in C^1(\Ns)$ touches $u$ from above at $(x_0,t_0) \in \Ns \subset \Omega$ we can consider the function
    \begin{equation*}
        \varphi^{\epsilon} = \varphi + \epsilon |x-x_0|^4 + |t-t_0|^2,
    \end{equation*}
    for which the touching is strict. Then, since $u$ is a subsolution, it follows that
    $$
    \frac{\partial \varphi}{\partial t}(x_0,t_0) = \frac{\partial \varphi^{\epsilon}}{\partial t}(x_0,t_0) \leq \Linf \varphi^{\epsilon}_{N}(x_0,t_0).
    $$
    Furthermore, we note that 
    \begin{equation*}
        \Linf \varphi_{N}(x_0,t_0) \leq \Linf \varphi^{\epsilon}_{N}(x_0,t_0) \leq \Linf \varphi_{N}(x_0,t_0)+\epsilon (\textup{diam}(_{N}))^{4-\alpha},
     \end{equation*}
   so  $\Linf \varphi^{\epsilon}_{N}(x_0,t_0) \to \Linf \varphi_{N}(x_0,t_0)$ as $\epsilon \to 0$. 
   Similarly, if $u$ is a supersolution and $\varphi \in C^1(\Ns)$ touches $u$ from below at $(x_0,t_0) \in  \Ns \subset \Omega$ we can make the touching strict by considering
        \begin{equation*}
        \varphi_{\epsilon} = \varphi- \epsilon |x-x_0|^4 - |t-t_0|^2.
    \end{equation*}
\end{rmk}
We are also going to define viscosity solutions to the stationary nonlocal infinity Laplace equation.
\begin{definition}\label{def:viscsolspace}
    Let $E \subset \Rn$ be a bounded open set and $\tilde{K} \in \R$. A function  $u: \Rn \to \R$ is a viscosity subsolution of 
    \begin{equation}\label{eq:usubKspace}
        \Linf u = \tilde{K},
    \end{equation}
    in $E \subset \Rn$ if the following conditions hold:
    \begin{itemize}
        \item $u$ is upper semicontinuous in $E$ and satisfies
         \begin{equation}\label{eq:growthcond2}
             |u(x)| \leq C(1+|x|^{\beta}), 
         \end{equation} 
         for some $C> 0$, $\beta < \alpha$. 
        \item let $x_0 \in E$ and let $N  \subset E$ be a neighborhood of $x_0$. If $\varphi \in C^1(N)$ such that $\varphi(x_0)=u(x_0)$ and $\varphi \geq u$ in $N$, then
    $$
    \mathcal{L}_{\infty}\varphi_N(x_0) \geq \tilde{K}, 
    $$
    where
\begin{align*}
     \varphi_N= \begin{cases}
                      \varphi & \textup{if  } x \in N,\\
                      u & \textup{if  } x \in \ \Rn   \setminus N.
            \end{cases}
 \end{align*} 
\end{itemize}
We say that $u$ is a viscosity supersolution in $E$ if $-u$ is a viscosity subsolution in $E$. Finally, $u$ is a viscosity solution in $E$ if it is a viscosity subsolution and a viscosity supersolution.
\end{definition}

Sometimes $N= B_r(x_0)$. In that case, we will use the notation
\begin{align}\label{eq:defvarphirspace}
     \varphi_r= \begin{cases}
                      \varphi & \textup{for  }  x \in B_r(x_0),\\
                      u & \textup{for  } x \in \ \Rn   \setminus  B_r(x_0).
            \end{cases}
 \end{align}
\begin{lem}\label{lem:Linfusc}
    Let $0<\alpha \leq 1$, $u$ be upper semicontinuous, $\Ns \subset \Omega$ and $\varphi \in C^1(\Ns)$. Then $\Linf \varphi_{N}$ is upper semicontinuous in $\Ns$, that is for $(x_0,t_0) \in \Ns$ we have that
    \begin{equation*}
        \limsup_{(x_k,t_k) \to (x_0,t_0)} \Linf \varphi_{N}(x_k,t_k) \leq \Linf \varphi_{N}(x_0,t_0),
    \end{equation*}
    where $\varphi_{N}$ is defined as in \eqref{eq:varphiN}. If $v$ is lower semicontinuous, then $\Linf \varphi_{N}$ is lower semicontinuous in $\Ns$.
\end{lem}

\begin{proof}
    We show the lemma for an upper semicontinuous function $u$. Then we can use the fact that $\Linf \varphi_{N} = - \Linf (-\varphi_{N})$ to see that $\Linf \varphi_{N}$ is lower semicontinuous if $v$ is lower semicontinuous. For $(x,t) \in \Ns$ we define
    \begin{equation*}
        L^+_{N}(x,t) = \sup_{y \in N, y\neq x} \frac{\varphi(y,t)-\varphi(x,t)}{|y-x|^{\alpha}} \quad \textup{and } \quad L^-_{N}(x,t) = \inf_{y \in N, y\neq x} \frac{\varphi(y,t)-\varphi(x,t)}{|y-x|^{\alpha}},
    \end{equation*}
and
    \begin{equation*}
        L^+_{N^c}(x,t) = \sup_{y \in \Rn \setminus N} \frac{u(y,t)-\varphi(x,t)}{|y-x|^{\alpha}} \quad \textup{and } \quad L^-_{N^c}(x,t) = \inf_{y \in \Rn \setminus N} \frac{u(y,t)-\varphi(x,t)}{|y-x|^{\alpha}}.
    \end{equation*}
    Furthermore, we note that
    \begin{align*}
        \Lpos \varphi_{N}(x,t) &= \max(L^+_{N^c}(x,t), (L^+_{N}(x,t)), \\
        \Lneg \varphi_{N}(x,t) &= \min(L^-_{N^c}(x,t), (L^-_{N}(x,t)). 
    \end{align*}
From Lemma 3.5 in \cite{CLM2012} we know that $L^{\pm}_{N}(x,t) \in C(\Ns)$. We are now going to show that $L^{\pm}_{N^c}(x,t)$ are upper semicontinuous. This is just a simple consequence of the fact that $u$ is upper semicontinuous. Since $u$ is upper semicontinuous and $\varphi$ is continuous, the following inequality is true for all $y \in \Rn \setminus N$ and $(x_k,t_k) \to (x_0,t_0)$.
\begin{equation*}
    \limsup_{k \to \infty} L^-_{N^c}(x_k,t_k) \leq \frac{u(y,t_0)-\varphi(x_0,t_0)}{|y-x_0|^{\alpha}}.
\end{equation*}
Therefore,
\begin{equation*}\label{eq:notinfty}
    \limsup_{k \to \infty} L^-_{N^c}(x_k,t_k) \leq L^-_{N^c}(x_0,t_0).
    \end{equation*}
 Thus, $L^-_{N^c}(x,t)$ is upper semicontinuous, and since $\Lneg \varphi(x,t)$ is the minimum of an upper semicontinuous function and a continuous function, it is also upper semicontinuous. 

Next, we turn to $L^+_{N^c}(x,t)$. If $L^+_{N^c}(x,t)$ is not upper semicontinuous there exists a sequence $(x_k,t_k) \to (x_0,t_0)$ such that
\begin{equation}\label{eq:contradiction}
        \lim_{k \to \infty} L^+_{N^c}(x_k,t_k) > L^+_{N^c}(x_0,t_0).
\end{equation}
It follows from Remark \ref{rmk:posneg}  and \eqref{eq:contradiction} that for large $k$, $ L^+_{N^c}(x_k,t_k)$ is attained for some $y_k \in \Rn \setminus N$, and $y_k \to \tilde{y}$. It follows from the upper semicontinuity of $u$ and the definition of $L^+_{N^c}(x_0,t_0)$ that
\begin{equation*}
     \lim_{k \to \infty} L^+_{N^c}(x_k,t_k) \leq \frac{u(\tilde{y},t_0)-\varphi(x_0,t_0)}{|\tilde{y}-x_0|^{\alpha}} \leq L^+_{N^c}(x_0,t_0).
\end{equation*}
However, this is a contradiction to \eqref{eq:contradiction}. Thus, $L^+_{N^c}(x,t)$ is also upper semicontinuous, and it follows that $\Lpos \varphi_{N}$ is upper semicontinuous. Therefore, $\Linf \varphi_{N}$ is upper semicontinuous, which concludes the proof of the lemma.
\end{proof}

\begin{lem}\label{lem:mincone}
    Let $0<\alpha \leq 1$. The function $\varphi(x) = \min(|x-x_0|^2, R^2)$ satisfies

    \begin{equation}\label{eq:Linfsub}
        \Linf \varphi(x)=   \left(R+|x-x_0| \right) \left(R-|x-x_0| \right)^{1-\alpha} - |x-x_0|^{2-\alpha} h(\alpha),
    \end{equation}
     at each $x \in B_R(x_0)$, where 
     $$
     h(\alpha) = \left(1+\frac{\alpha}{2-\alpha} \right)\left(1-\frac{\alpha}{2-\alpha} \right)^{1-\alpha}.
     $$
\end{lem}
\begin{proof}
    Take $x \in B_R(x_0)$. Then 
    \begin{equation*}
        \Linf \varphi(x)= \sup_{y \in \Rn} \frac{\min(|y-x_0|^2, R^2)-|x-x_0|^2}{|x-y|^{\alpha}} + \inf_{y \in \Rn} \frac{\min(|y-x_0|^2, R^2)-|x-x_0|^2}{|x-y|^{\alpha}}.
    \end{equation*}
    First, we note that for the supremum we need to only consider $y \in \Rn$ such that $\min(|y-x_0|^2, R^2) > |x-x_0|^2$,  because otherwise the term will be zero or negative. Furthermore, since the numerator only depends on $\dist(y,x_0)$ we can take the supremum over spheres centered at $x_0$, $\partial B_s(x_0)$ for all $s > |x-x_0|$, that is
    \begin{equation*}
          \sup_{y \in \Rn} \frac{\min(|y-x_0|^2, R^2)-|x-x_0|^2}{|x-y|^{\alpha}}= \sup_{y \in \partial B_s(x_0), s> |x-x_0|} \frac{\min(|y-x_0|^2, R^2)-|x-x_0|^2}{|x-y|^{\alpha}}.
    \end{equation*}
    In addition, in order to make the denominator as small as possible we need to only consider points, $y$, on the line with direction $(x-x_0)/|x-x_0|$, i.e., $y=x_0+r(x-x_0)$ with $r>1$. Furthermore, for points on the line such that $r \geq  R/|x-x_0|$, we note that the denominator will be constant and equal to $R^2$ while the numerator will keep growing as $r$ gets larger. We conclude that the supremum must be attained for some $y=x_0+r(x-x_0)$ where $1<r\leq R/|x-x_0|$. Thus, we see that
     \begin{align*}
      \sup_{y \in \partial B_s(x_0), s> |x-x_0|} \frac{\min(|y-x_0|^2, R^2)-|x-x_0|^2}{|x-y|^{\alpha}} &= \sup_{1 < r \leq R/|x-x_0|} |x-x_0|^{2-\alpha} \frac{r^2-1}{(r-1)^{\alpha}}\\
          &= |x-x_0|^{2-\alpha}  \sup_{1 < r \leq R/|x-x_0|} (r+1)(r-1)^{1-\alpha} \\
          &= |x-x_0|^{2-\alpha} \left(R/|x-x_0| +1\right)\left(R/|x-x_0| -1\right)^{1-\alpha}\\
          &= \left(R+|x-x_0| \right)  \left(R-|x-x_0| \right)^{1-\alpha}.
     \end{align*}
     Next, we turn to the infimum and we note that with a similar argument as above it follows that we only need to consider points on the line $y =x_0+r(x-x_0)$ where $0<r<1$. Thus,
     \begin{align*}
         \inf_{y \in \Rn} \frac{\min(|y-x_0|^2, R^2)-|x-x_0|^2}{|x-y|^{\alpha}} &= \inf_{0<r<1} |x-x_0|^{2-\alpha}\frac{r^2-1}{(1-r)^{\alpha}} \\
         &= - |x-x_0|^{2-\alpha} \sup_{0<r<1} (1-r)^{1-\alpha}(1+r).
     \end{align*}
     We note that for $\alpha=1$ the supremum is 2. Next, we set $g(r) = (1-r)^{1-\alpha}(1+r)$. For $0<\alpha<1$ we note that $g(0)=1$ and $g(1)=0$ and 
     \begin{align*}
         g'(r)&= (1-r)^{1-\alpha} -(1+r)(1-\alpha)(1-r)^{-\alpha} \\
         &=(1-r)^{-\alpha}\left(1-r -(1-\alpha)(1+r) \right) \\
         &= (1-r)^{-\alpha}\left(r(\alpha-2)+\alpha \right).
     \end{align*}
     Thus, $r=\frac{\alpha}{2-\alpha}$ gives an extreme point and we define 
     $$
    h(\alpha):= g\left(\frac{\alpha}{2-\alpha}\right)= \left(1+\frac{\alpha}{2-\alpha} \right)\left(1-\frac{\alpha}{2-\alpha} \right)^{1-\alpha}.
    $$
    From Lemma \ref{lem:simple} we know that $h(\alpha)>1$ for $0<\alpha \leq1$, and therefore $g\left(\frac{\alpha}{2-\alpha}\right)$ is a maximum. Thus,
    \begin{equation*}
         \inf_{y \in \Rn} \frac{\min(|y-x_0|^2, R^2)-|x-x_0|^2}{|x-y|^{\alpha}}=  - |x-x_0|^{2-\alpha} h(\alpha).
    \end{equation*}
     By adding the supremum and the infimum for the different cases \eqref{eq:Linfsub} follows. This proves the lemma.
     \end{proof}

\subsection{The comparison principle}\label{sec:comparisonprinciple}
The aim of this section is to prove a comparison principle for viscosity solutions to \eqref{eq:viscsol}. We start by proving that if we can touch a subsolution of \eqref{eq:usubKspace} from above at $x_0 \in E \subset \Rn$ with some test function $\varphi$, then $\Linf u(x_0)$ exists in the sense that $\Lneg u(x_0) < \infty$, and we can treat the viscosity subsolution as a classical subsolution. Note that we only consider $0<\alpha<1$ in this lemma. 
\begin{lem}\label{lem:Linfpointwisespace}
      Let $0<\alpha <1$. Assume that $u$ is a subsolution to \eqref{eq:usubKspace} in $E\subset \Rn$
     in the viscosity sense. Furthermore, assume that $x_0 \in E$, $N \subset E$ is a neighborhood of $x_0$ and $\varphi \in C^1(N)$ touches $u$ from above at $x_0$. Then 
     \begin{equation*}
         \Linf u(x_0) \geq \tilde{K}.   
     \end{equation*}
\end{lem}

\begin{proof}
We start by choosing $r$ so small that $B_r(x_0) \subset N$. Since $u$ is a viscosity subsolution to \eqref{eq:usubKspace} we have that
\begin{equation}\label{eq:ineqstat}
    \Linf \varphi_r(x_0) \geq \tilde{K}.
\end{equation}   
Since $\Lneg \varphi_r(x_0) \leq 0$, and $\Lpos \varphi_{\tilde{r}}(x_0) \leq \Lpos \varphi_{r}(x_0)$ for $\tilde{r} < r$ we see that
\begin{align*}
    0 \leq -\Lneg \varphi_{\tilde{r}}(x_0) \leq -\tilde{K} + \Lpos \varphi_{\tilde{r}}(x_0) \leq -\tilde{K} + \Lpos \varphi_{r}(x_0)< \infty.
\end{align*}
 Thus, for fixed $r>0$ and a sequence $\{r_k\}$ such that $r_k \to 0$ when $k \to \infty$ we see that
 \begin{align*}
    0 \leq -\Lneg \varphi_{r_k}(x_0) \leq  -\tilde{K} + \Lpos \varphi_{r}(x_0)< \infty,
\end{align*}
 and it follows from the monotonicity of $r \to \Lneg \varphi_r$ that $-\Lneg \varphi_{r_k}(x_0) \to \tilde{C}$, where $0\leq \tilde{C} < \infty$. Thus,
  \begin{align}\label{eq:Dshortspace}
    \tilde{C} \leq   -\tilde{K} + \Lpos \varphi_{r}(x_0).
\end{align}
 
 Next, we take $\rho >0$. For $k$ so large that $\rho > r_k$ we note that
\begin{align*}
    \Lneg \varphi_{r_k}(x_0) = \inf_{y \in \Rn} \frac{\varphi_{r_k}(y)-\varphi_{r_k}(x_0)}{|y-x_0|^{\alpha}} \leq \inf_{y \in \Rn \setminus B_{\rho}(x_0)} \frac{u(y)-u(x_0)}{|y-x_0|^{\alpha}},
\end{align*}
and by letting $k \to \infty$ we obtain
\begin{equation*}
    \tilde{C} \geq - \inf_{y \in \Rn \setminus B_{\rho}(x_0)} \frac{u(y)-u(x_0)}{|y-x_0|^{\alpha}}.
\end{equation*}
Since this is true for all $\rho >0 $, we see that $\tilde{C} \geq - \Lneg u(x_0)$. Therefore, it follows from \eqref{eq:Dshortspace} that
\begin{equation}\label{eq:Drspace}
    -\Lneg u(x_0) \leq -\tilde{K} + \Lpos \varphi_{r}(x_0).
\end{equation}
Next, we study the limit of $\Lpos \varphi_{r}(x_0)$ as $r \to 0$. We have that
\begin{align}\label{eq:Lposvarphirspace}
    \Lpos \varphi_r(x_0) = \max\left(\sup_{y \in \Rn \setminus B_r(x_0) } \frac{u(y)-u(x_0)}{|y-x_0|^{\alpha}}, \sup_{y \in B_r(x_0)} \frac{\varphi(y)-\varphi(x_0)}{|y-x_0|^{\alpha}}\right).
\end{align}
At first we note that 
\begin{align*}
   \sup_{y \in \Rn \setminus B_r(x_0)} \frac{u(y)-u(x_0)}{|y-x_0|^{\alpha}} \geq 0,
\end{align*}
due to Remark \ref{rmk:posneg}. Furthermore, since $\varphi_r$ is Lipschitz continuous on compact sets it follows that
\begin{align*}
    \sup_{y \in B_r(x_0)} \frac{\varphi(y)-\varphi(x_0)}{|y-x_0|^{\alpha}} \leq \sup_{y \in \overline{B_r(x_0)}} C |y-x_0|^{1-\alpha} \leq C r^{1-\alpha}.
\end{align*}
 Thus, if we let $r \to 0$ in \eqref{eq:Lposvarphirspace} we see that
\begin{align*}
   \lim_{r \to 0} \Lpos \varphi_r(x_0) = \max\left(\Lpos u(x_0),0 \right) = \Lpos u(x_0),
\end{align*}
where we have used that $\Lpos u(x_0) \geq 0$ due to Remark \ref{rmk:posneg}. By sending $r$ to zero in \eqref{eq:Drspace} we obtain
 \begin{equation*}
-\Linf u(x_0) \leq -\tilde{K},
 \end{equation*}
 which proves the lemma.
\end{proof}

The corresponding result holds naturally for the viscosity subsolutions of \eqref{eq:viscsol}. It is formulated in the following corollary.
\begin{cor}\label{cor:Linfpointwise}
     Let $0<\alpha <1$. Assume that $u$ is a subsolution to \eqref{eq:viscsol}
     in the viscosity sense. Furthermore, assume that $(x_0,t_0) \in \Omega$,  $\Ns $ is a neighbourhood of $(x_0,t_0)$ such that $ \Ns \subset \Omega$, and $\varphi \in C^1(\Ns)$ touches $u$ from above at $(x_0,t_0)$. Then 
     \begin{equation*}
        \frac{\partial \varphi(x_0,t_0)}{\partial t} - \Linf u(x_0,t_0) \leq K.   
     \end{equation*}
\end{cor}
\begin{proof}
We start by choosing $r$ so small that $B_r(x_0) \subset N$. Since $u(x,t)$ is a viscosity subsolution to \eqref{eq:viscsol} we have that
\begin{equation*}
    \frac{\partial \varphi(x_0,t_0) }{\partial t}-\Linf \varphi_r(x_0,t_0) \leq K,
\end{equation*}   
or equivalently that
\begin{equation*}
    -\Linf \varphi_r(x_0,t_0) \leq \underbrace{K-  \frac{\partial \varphi(x_0,t_0) }{\partial t}}_{:= \tilde{K}}.
\end{equation*}
The rest of the proof follows from the exact same arguments as in Lemma \ref{lem:Linfpointwisespace}, starting from \eqref{eq:ineqstat}.
\end{proof}

 \begin{prop}\label{lem:compprinc}
    Let $0<\alpha <1$. Furthermore, let $v$ be a bounded viscosity supersolution and $u$ a bounded viscosity subsolution of \eqref{eq:Luf} in $\Omega$ such that $u \leq v$ in $\RnpO$, and
    $$
   \limsup_{(x,t) \to (x_0,t_0)} u(x,t) \leq  \liminf_{(x,t) \to (x_0,t_0)} v(x,t) \quad \textup{for } (x,t) \in  \Omega \textup{ and } (x_0,t_0) \in \partial \Omega \cap (\Rn \times[T_0,T)).
    $$
    Suppose also that $v$ and $u$ are lower semicontinuous and upper semicontinuous in $\Rnp$, respectively, and that
        \begin{equation}\label{eq:limitatinf}
               \limsup_{(x,t) \textup{ such that } |x| \to \infty} u(x,t) \leq  \liminf_{(x,t) \textup{ such that } |x| \to \infty} v(x,t).
        \end{equation}
    Then $u(x,t) \leq v(x,t)$ for all $(x,t) \in \Omega$. 
\end{prop}

The proof is almost identical to a combination of the first part of Theorem 4.1 in \cite{KKL2019} and Proposition 4.10 in \cite{HL2016} and we will therefore only provide a sketch. 
\begin{proof}
   Take $\delta >0$. In order to be able to work with a strict subsolution we define $\tilde{u}= u+ \frac{\delta}{t-T}$, so that $\tilde{u}$ becomes a viscosity subsolution of the equation
    \begin{equation*}
        \frac{\partial v}{\partial t}-\Linf v = -\frac{\delta}{(t-T)^2},
    \end{equation*}
    in $\Omega \cap \Rn \times [T_0, T-\gamma]$, for any $\gamma < T-T_0$. We let 
    \begin{equation*}
        K = \inf_{t \in (T_0,T)} \frac{\delta}{(t-T)^2}=\frac{\delta}{(T_0-T)^2},
    \end{equation*}
    and note that $\tilde{u}$ solves
        \begin{equation*}
        \frac{\partial \tilde{u}}{\partial t}-\Linf \tilde{u} \leq -K,
    \end{equation*}
    in the viscosity sense. Since $\tilde{u} < u$, the boundary conditions still hold when we replace $u$ with $\tilde{u}$ and we note that $\tilde{u}(x,t)-v(x,t)\to -\infty$ as $t \to T$ since $u$ and $v$ are bounded as $t \to T$. We let 
    \begin{equation*}
        M= \sup_{\Omega} (\tilde{u}-v),
    \end{equation*}
    and for the sake of contradiction we assume that $M>0$. We define
    \begin{equation*}
        \Psi_{\epsilon}(x,y,t,\tau)= \tilde{u}(x,t)-v(y,\tau)-\frac{|x-y|^2+|t-\tau|^2}{\epsilon},
    \end{equation*}
    and 
    \begin{equation*}
        M_{\epsilon} = \sup_{\Rn \times [T_0,T] \times \Rn \times [T_0,T]} \Psi_{\epsilon}(x,y,t,\tau). 
    \end{equation*}
    By the semicontinuity and \eqref{eq:limitatinf} it follows that $M_{\epsilon}$ is attained. We call the point where $M_{\epsilon}$ for $(x_{\epsilon}, y_{\epsilon}, t_{\epsilon}, \tau_{\epsilon})$. Moreover, from Proposition 3.7 in \cite{CIL1992} it follows that there is a subsequence $\{\epsilon_j \}$ that we will continue to denote by $\epsilon$ such that $x_{\epsilon}\to x^*, y_{\epsilon}\to x^*, t_{\epsilon}\to t^*$ and $\tau_{\epsilon}\to t^*$, where $(x^*, t^*) \in \Omega$. Therefore, for $\epsilon $ small enough we may assume that $(x_{\epsilon}, t_{\epsilon}), (y_{\epsilon}, \tau_{\epsilon}) \in \Omega$. In addition,
    \begin{equation*}
        \lim_{\epsilon \to 0} \tilde{u}(x_{\epsilon}, t_{\epsilon})-v(y_{\epsilon}, \tau_{\epsilon}) =\tilde{u}(x^*,t^*)-v(x^*,t^*)= \sup_{\Omega}(\tilde{u}-v) = M.
    \end{equation*}
From the definition, the function
\begin{equation*}\label{eq:tildeutouch}
    \varphi_{\epsilon}(x,t)= \tilde{u}(x_{\epsilon},t_{\epsilon})+\frac{|x-y_{\epsilon}|^2+|t-\tau_{\epsilon}|^2}{\epsilon} - \frac{|x_{\epsilon}-y_{\epsilon}|^2+|t_{\epsilon}-\tau_{\epsilon}|^2}{\epsilon},
\end{equation*}
touches $\tilde{u}$ from above at $(x_{\epsilon}, t_{\epsilon})$ and the function 
\begin{equation*}\label{eq:vtouch}
    \phi_{\epsilon}(y,\tau)= v(y_{\epsilon},\tau_{\epsilon})+\frac{|x_{\epsilon}-y_{\epsilon}|^2+|t_{\epsilon}-\tau_{\epsilon}|^2}{\epsilon} - \frac{|x_{\epsilon}-y|^2+|t_{\epsilon}-\tau|^2}{\epsilon},
\end{equation*}
 touches $v$ from below at $(y_{\epsilon}, \tau_{\epsilon})$.
 According to Lemma \ref{cor:Linfpointwise} and the fact that $\frac{d\varphi_{\epsilon}}{\partial t}(x_{\epsilon}, t_{\epsilon})= \frac{d\phi_{\epsilon}}{\partial \tau}(y_{\epsilon}, \tau_{\epsilon})$ we have that
\begin{equation*}
    \Linf \tilde{u}(x_{\epsilon},t_{\epsilon}) > \Linf v(y_{\epsilon},\tau_{\epsilon}).
\end{equation*}
 Furthermore, since $\Psi_{\epsilon}(x_{\epsilon}, y_{\epsilon}, t_{\epsilon}, \tau_{\epsilon}) \geq \Psi_{\epsilon}(x_{\epsilon}+z, y_{\epsilon}+z, t_{\epsilon}, \tau_{\epsilon})$ for each $z \in \Rn$ it also follows that
\begin{equation}\label{eq:Linfcomp}
    \tilde{u}(x_{\epsilon}+z,t_{\epsilon})-\tilde{u}(x_{ \epsilon},t_{\epsilon}) \leq v(y_{\epsilon}+z, \tau_{\epsilon})-v(y_{\epsilon},\tau_{\epsilon}),
\end{equation}
and thus,
\begin{equation*}
       \Linf \tilde{u}(x_{\epsilon},t_{\epsilon}) \leq \Linf v(y_{\epsilon},\tau_{\epsilon}),
\end{equation*}
which is a contradiction. Therefore, $\tilde{u}  \leq v$. We can now let $\delta \to 0$, which proves the proposition.
\end{proof}

\section{The Perron method}\label{sec:Perron}
In this section we are going to construct solutions to \eqref{eq:Luf} using Perron's method. 

\begin{lem}\label{lem:supsubissub}
    Let $\mathcal{A}$ be a family of viscosity subsolutions (resp. viscosity supersolutions) of \eqref{eq:Luf} and define 
    \begin{equation}\label{eq:sup_u}
        u(x,t):= \sup_{w \in \mathcal{A}} w(x,t) \quad (\textup{resp. } v(x,t):= \inf_{w \in \mathcal{A}} w(x,t) ).
    \end{equation}
Then $u^*$ (resp $v_*$) is a viscosity subsolution (resp. supersolution) of $\eqref{eq:Luf}$ in $\Omega$, where $u^*$ and $v_*$ are defined in \eqref{eq:u_upp} and \eqref{eq:u_low}, respectively.
\end{lem}
\begin{proof}
    We provide the proof for subsolutions. We consider $(x_0, t_0) \in \Omega$, a neighborhood $\Ns$ of $(x_0,t_0)$ such that $ \Ns \subset \Omega$ and let we let $\varphi \in C^1(\Ns)$ be a test function that touches $u^*$ strictly from above at $(x_0,t_0)$. From Lemma \ref{lem:strictmax} we know that that there exists a cylinder $Q_r=B_r(x_0,t_0) \times (t_0-r,t_0+r)$, such that $Q_r \subset \subset \Ns$, and points $(x_k,t_k) \in Q_r$, such that $(x_k,t_k)\to (x_0, t_0)$, and functions $u_k \in \mathcal{A}$ such that 
  \begin{equation}\label{eq:supoutsideuk}
         \sup_{Q_r } (u_k-\varphi) = u_k(x_k,t_k)-\varphi(x_k,t_k),
    \end{equation}
     and $u_k(x_k,t_k) \to u^*(x_0, t_0)$. We note that $u_k(x_k,t_k)\leq u^*(x_k,t_k) < \varphi(x_k,t_k)$ for each $(x_k,t_k) \neq (x_0,t_0)$. However, from \eqref{eq:supoutsideuk} it follows that the function $\varphi-M_k$ touches $u_k$ from above at $(x_k,t_k)$, where
     $$
     M_k = \inf_{Q_r} (\varphi-u_k) = -\sup_{Q_r } (u_k-\varphi) = \varphi(x_k,t_k)-u_k(x_k,t_k).
     $$
     We next define $\varphi^k=\varphi-M_k$ and
    \begin{align*}
     \varphi_{r }^k= \begin{cases}
                      \varphi-M_k & \textup{if  } \ (x,t) \in Q_r ,\\
                      u_k & \textup{if  } (x,t) \in \ \Rnp \setminus Q_r,
            \end{cases}
 \end{align*}
and recall that 
\begin{align*}
     \varphi_{r }= \begin{cases}
                      \varphi & \textup{if  } \ (x,t) \in Q_r ,\\
                      u^* & \textup{if  } (x,t) \in \ \Rnp \setminus Q_r.
            \end{cases}
 \end{align*}
We introduce the notation
    \begin{equation*}
        L^+_{r}(x,t) = \sup_{y \in B_r(x_0)} \frac{\varphi(y,t)-\varphi(x,t)}{|y-x|^{\alpha}} \quad \textup{and } \quad L^-_{r}(x,t) = \inf_{y \in B_r(x_0)} \frac{\varphi(y,t)-\varphi(x,t)}{|y-x|^{\alpha}},
    \end{equation*}
and
    \begin{equation*}
        L^+_{r^c}(x,t) = \sup_{y \in \Rn \setminus B_r(x_0)} \frac{u_k(y,t)-\varphi(x,t)+M_k}{|y-x|^{\alpha}} \quad \textup{and } \quad L^-_{r^c}(x,t) = \inf_{y \in \Rn \setminus B_r(x_0)} \frac{u_k(y,t)-\varphi(x,t) +M_k}{|y-x|^{\alpha}},
    \end{equation*}
    and we note that for $(x,t) \in Q_{r/2}$ we have that
\begin{align*}
     L^+_{r^c}(x,t) &\leq \sup_{y \in \Rn \setminus B_r(x_0)} \frac{u_k(y,t)-\varphi(x,t)}{|y-x|^{\alpha}} + \sup_{y \in \Rn \setminus B_r(x_0)} \frac{M_k}{|y-x|^{\alpha}} \\
     &= \sup_{y \in \Rn \setminus B_r(x_0)} \frac{u_k(y,t)-\varphi(x,t)}{|y-x|^{\alpha}} + \frac{M_k}{\inf_{y \in \Rn \setminus B_r(x_0)}|y-x|^{\alpha}} \\
     & \leq \sup_{y \in \Rn \setminus B_r(x_0)} \frac{u^*(y,t)-\varphi(x,t)}{|y-x|^{\alpha}} + \frac{2^{\alpha} M_k}{r^{\alpha}}.
\end{align*}
Furthermore,
\begin{align*}
    L^-_{r^c}(x,t) &\leq \inf_{y \in \Rn \setminus B_r(x_0)} \frac{u_k(y,t)-\varphi(x,t)}{|y-x|^{\alpha}} + \sup_{y \in \Rn \setminus B_r(x_0)} \frac{M_k}{|y-x|^{\alpha}} \\
     &= \inf_{y \in \Rn \setminus B_r(x_0)} \frac{u_k(y,t)-\varphi(x,t)}{|y-x|^{\alpha}} + \frac{M_k}{\inf_{y \in \Rn \setminus B_r(x_0)}|y-x|^{\alpha}} \\
     & \leq \inf_{y \in \Rn \setminus B_r(x_0)} \frac{u^*(y,t)-\varphi(x,t)}{|y-x|^{\alpha}} + \frac{2^{\alpha} M_k}{r^{\alpha}}.
\end{align*}
For $(x,t) \in Q_r$ we have that
    \begin{align*}
        \Lpos \varphi_{r}^k(x,t) &= \max(L^+_{r^c}(x,t), L^+_{r}(x,t)), \\
        \Lneg \varphi_{r}^k(x,t) &= \min(L^-_{r^c}(x,t), L^-_{r}(x,t)).
    \end{align*}
Thus, for $k$ large enough such that $(x_k,t_k) \in Q_{r/2}$ it follows that
\begin{equation*}
    \Linf \varphi^k_{r}(x_k,t_k) \leq \Linf \varphi_{r}(x_k,t_k)+ \frac{2^{\alpha} M_k}{r^{\alpha}},
\end{equation*}
where the last term goes to zero $k \to \infty$ since $M_k$ goes to zero as $k \to \infty$. Since each $u_k$ is a subsolution in the viscosity sense it follows that
    \begin{align*}
       0 &\geq \lim_{k \to \infty} \left(\frac{\partial \varphi^k}{\partial t}(x_k, t_k) - \Linf \varphi_{r }^k(x_k, t_k) \right)\\
       &\geq \lim_{k \to \infty} \left(\frac{\partial \varphi}{\partial t}(x_k, t_k) - \Linf \varphi_{r }(x_k, t_k) -\frac{2^{\alpha} M_k}{r^{\alpha}}\right)\\
       &\geq \frac{\partial \varphi}{\partial t}(x_0, t_0) - \Linf \varphi_{r}(x_0, t_0)\\
       & \geq  \frac{\partial \varphi}{\partial t}(x_0, t_0) - \Linf \varphi_{N }(x_0, t_0),
    \end{align*}
where we used Lemma \ref{lem:Linfusc} in the third inequality. The last inequality follows from the fact that $\Linf \varphi_{r }(x_0, t_0) \leq  \Linf \varphi_{N }(x_0, t_0)$ since $\varphi \geq u^*$ in $\Ns$. Thus, we have proved that $u^*$ is a viscosity subsolution.
\end{proof}

\subsection{Upper and lower Perron solutions}

\begin{definition}\label{def:upperlowerclass}
Let $g:\RnpO \to \R$ be a continuous, bounded function with a limit at infinity. The upper class $\mathcal{U}_g$ consists of all bounded functions $v$ such that 
\begin{itemize}
    \item $v:\Rnp \to (-\infty,\infty)$ is a viscosity supersolution of \eqref{eq:Luf} and $v \in C(\Omega)$,
    \item $\liminf_{\xi \to \xi_0} v(\xi) \geq g(\xi_0)$ for $\xi \in \Omega$, $\xi_0 \in \partial \Omega$, 
    \item  $v \geq g$ in $\Rnp \setminus \Omega$.
\end{itemize} 
 The lower class $\mathcal{L}_g$ consists of all bounded functions $u$ such that 
\begin{itemize}
    \item $u:\Rnp \to (-\infty,\infty)$ is a viscosity subsolution of \eqref{eq:Luf} and $u \in C(\Omega)$,
    \item$\limsup_{\xi \to \xi_0} u(\xi) \leq g(\xi_0)$ when $\xi \in \Omega$, $\xi_0 \in \partial \Omega$,
    \item  $u \leq g$ in $\Rnp \setminus \Omega$.
\end{itemize}
\end{definition}

We note that $\ug$ and $\lcg$ are non-empty since $-\|g\|_{L^{\infty}(\RnpO)} \in \lcg$ and $\|g\|_{L^{\infty}(\RnpO)} \in \ug$. We define the upper Perron solution $\ups$ and the lower Perron solution $\lps$ as
\begin{equation*}
    \ups(z) = \inf_{v \in \ug} \{ v(z)\}, \quad \quad \lps(z) = \sup_{u \in \lcg} \{ u(z) \}.
\end{equation*}

 For $0<\alpha <1$ it follows directly from the comparison principle that $\lps \leq \ups$. If there exists a viscosity solution $h_g$ such that $h_g(\xi_0)=g(\xi_0)$ for $\xi_0 \in \RnpO$, and $\lim_{\xi \to \xi_0} h_g(\xi) = h_g(\xi_0)$ for $\xi \in \Omega$ and $\xi_0 \in \partial \Omega$, it follows that $h_g \in \mathcal{U}_g \cap \mathcal{L}_g$. Thus, by definition $\ups \leq h_g \leq \lps$ and therefore $\lps = h_g=\ups$. In the next lemma we show that both the upper and the lower Perron solution are equal to the boundary function $g$ on $\RnpO$. The technique is similar to the proof of Theorem 9.2 in \cite{bbk2024perron}.
\begin{lem}\label{lem:Perronbdy}
    We define 
    \begin{align*}
        \underline{Q}_g &= \sup_{u \in \tilde{\mathcal{L}}_g} u \quad  \textup{where } \tilde{\mathcal{L}}_g = \{ u \in \lcg : u=g \textup{ on } \RnpO\},\\
        \overline{Q}_g &= \inf_{v \in \tilde{\mathcal{U}}_g} v \quad  \textup{where } \tilde{\mathcal{U}}_g = \{ v \in \ug : v=g \textup{ on } \RnpO\}.
    \end{align*}
    Then $\lps = \underline{Q}_g$ and $\ups = \overline{Q}_g$.
\end{lem}
\begin{proof}
    We show the result for $\lps$. For $u \in \lcg$ we define
    \begin{align*}
         \tilde{u}= \begin{cases}
            u & \textup{in  } \ \Omega,  \\
            g & \textup{on  } \ \RnpO.
        \end{cases}
    \end{align*}
    We are going to show that $\tilde{u} \in \lcg$. The second and third conditions in Definition \ref{def:upperlowerclass} are trivially satisfied for $\tilde{u}$. We just need to show that $\tilde{u}$ is a subsolution to \eqref{eq:Luf}.  We consider $(x_0, t_0) \in \Omega$, and $\Ns$ such that $(x_0,t_0) \in \Ns \subset \Omega$, and let we let $\varphi \in C^1(\Ns)$ be a test function that touches $\tilde{u}$ strictly from above at $(x_0,t_0)$. Since $\varphi$ also touches $u$ from above at $(x_0,t_0)$ we define 
    \begin{align*}
     \varphi_N^u= \begin{cases}
                      \varphi & \textup{if  } \ (x,t) \in N,\\
                      u & \textup{if  } (x,t) \in \ \Rn   \setminus N.
            \end{cases}
 \end{align*}
 We note that $\varphi_N \geq \varphi^u_N$ and therefore $\Linf \varphi_N(x_0,t_0) \geq \Linf \varphi^u_N(x_0,t_0)$. Since $u$ is a subsolution it follows that 
 \begin{align*}
     \frac{\partial \varphi_N}{\partial t}(x_0,t_0)-\Linf \varphi_N(x_0,t_0) \leq \frac{\partial \varphi_N^u}{\partial t}(x_0,t_0)-\Linf \varphi_N^u(x_0,t_0) \leq 0.
 \end{align*}
 Thus, $\tilde{u}$ is a subsolution so $\tilde{u} \in \lcg$. Therefore,
 \begin{equation*}
     \lps = \sup_{u \in \lcg} u = \sup_{u \in \lcg} \tilde{u} = \sup_{v \in \tilde{\mathcal{L}}_g } v 
     =   \underline{Q}_g.
\end{equation*}
\end{proof}
 Note that from the comparison principle and the lemma it now follows that 
 $$\lps \leq \|g\|_{L^{\infty}(\RnpO)} \textup{ and } \ups \geq -\|g\|_{L^{\infty}(\RnpO)},$$ in $\Omega$.
\begin{lem}\label{lem:subissup}
$\lps$ (resp. $\ups$) is a viscosity supersolution (resp. subsolution) to $\eqref{eq:Luf}$ such that $\lps=g$ (resp. $\ups =g$) on $\RnpO$.
\end{lem}

\begin{proof}
    We prove the lemma for $\lps$. Since $\lps$ is the pointwise supremum of a collection of continuous functions, it is lower semicontinuous. Assume that $\lps$ is not a viscosity supersolution. Then there exists some $(x_0,t_0) \in \Omega $, some neighborhood $\Ns \subset \Omega$, such that $(x_0,t_0)\in \Ns$, and a test function $\varphi \in C^1(\Ns)$ that touches $\lps$ strictly from below at $(x_0,t_0)$ such that
    \begin{equation*}
        \frac{\partial \varphi}{\partial t}(x_0,t_0) - \Linf \varphi_{N}(x_0,t_0) < 0.
    \end{equation*}
    We note that we may assume that $\varphi \in C(\overline{N} \times \overline{I})$, because otherwise we can just consider a smaller neighbourhood. We next define
    \begin{align*}
   \varphi^{\epsilon} =  \begin{cases}
\varphi &  \textup{ in }  \Ns, \\
 \lps-\epsilon &  \textup{ in }  (N  \times I)^c,
\end{cases}
    \end{align*}
and note that 
\begin{equation*}
    \frac{\partial \varphi}{\partial t}(x_0,t_0) - \Linf \varphi^{\epsilon}(x_0,t_0) < 0,
\end{equation*}
for some small $\epsilon >0$.  From Lemma \ref{lem:Linfusc} we recall
    that since $\lps-\epsilon$ is lower semicontinuous, $\Linf \varphi^{\epsilon}(x,t)$ is lower semicontinuous in $\Ns$. Therefore, $\frac{\partial \varphi}{\partial t}(x,t) - \Linf \varphi^{\epsilon}(x,t)$ is upper semicontinuous in $\Ns$. Consequently, there exists some $\mathcal{B}_{r}(x_0, t_0) \subset \subset \Ns$ such that
 \begin{equation*}
   \frac{\partial \varphi}{\partial t}(x,t) - \Linf \varphi^{\epsilon}(x,t) < 0,  
   \end{equation*}
   for $(x,t)  \in \mathcal{B}_{r}(x_0,t_0)$. Since $\varphi^{\epsilon}<\lps$ in $\mathcal{B}_{r}^c(x_0,t_0)$, it follows that for each $z \in \mathcal{B}_{r}^c(x_0,t_0)  \cap \overline{\Omega}$ there exists a $f_z \in \lcg \cap \tilde{\mathcal{L}}_g$ such that $f_z(z) > \varphi^{\epsilon}(z)$. Due to the continuity of the functions $f_{z}$ and $\varphi^{\epsilon}$, there exists a radius $s_{z}>0$ such that $f_{z} > \varphi^{ \epsilon}$ in $ \mathcal{B}_{s_{z}}(z)$. We note that we can cover the whole $\mathcal{B}_{r}^c(x_0,t_0)  \cap \overline{\Omega}$ with similar balls. Since $\mathcal{B}_{r}^c(x_0,t_0)  \cap \overline{\Omega}$ is a compact set, we can choose a finite subcover i.e., there exists $\{ z_1,\hdots, z_M \}$ and $\{s_{z_1}, \hdots, s_{z_M} \}$ such that 
    $$
    \mathcal{B}_{r}^c(x_0,t_0)  \cap \overline{\Omega} \subset \cup_{i=1}^M \mathcal{B}_{s_{z_i}}(z_i).
    $$
    We now choose
    \begin{equation*}
        f(x,t) = \max\{f_{z_1}(x,t), f_{z_2}(x,t), \hdots, f_{z_M}(x,t) \},
    \end{equation*}
    and note that by definition $f>\varphi^{\epsilon}$ in $\mathcal{B}_r^c(x_0,t_0) \cap \overline{\Omega}$ and $f-\varphi^{\epsilon}= \epsilon$ in $\Omega^c$. Thus, $f > \varphi^{\epsilon}$ on $\mathcal{B}_{r}^c(x_0,t_0)$. Furthermore, $f$ is a subsolution by Lemma \ref{lem:subissup}. We next define
    \begin{equation*}
        \delta = \min_{\Rnp \setminus \mathcal{B}_{r}(x_0,t_0)}(f-\varphi^{\epsilon}),
    \end{equation*}
    and note that $\delta >0$ because of the definition of $f$. We consider
    \begin{align*}
     w= \begin{cases}
            \max(f, \varphi^{\epsilon}+\delta) & \textup{in  } \ \mathcal{B}_r(x_0,t_0),  \\
            f & \textup{in  } \ \Rnp \setminus \mathcal{B}_r(x_0,t_0).
        \end{cases}
    \end{align*}
    We will now show that $w$ is a subsolution. Let $\xi_0 \in \Omega$ and let $\tilde{N} \times \tilde{I}$ be a neighbourhood of $\xi_0$. Consider $\psi \in C^1(\tilde{N} \times \tilde{I})$ that touches $w$ from above at $\xi_0$. If $w(\xi_0)=f(\xi_0)$, then since $w \geq f$, $\psi$ touches $f$ from above at $\xi_0$. Since $f$ is a subsolution, the subsolution property holds at $\xi_0$. If, on the other hand, $w(\xi_0)= \varphi^{\epsilon}(\xi_0)+\delta$, then $\psi$ touches $\varphi^{\epsilon} + \delta$ from above since $w \geq \varphi^{\epsilon}+\delta$ by the choice of $w$ and $\delta$. Since $\varphi^{\epsilon}$ is a subsolution in $\mathcal{B}_r(x_0,t_0)$, the subsolution property holds at $\xi_0$ also in this case. It is now clear that $w \in \lcg$, and therefore, $w \leq \lps$. However, since $\varphi^{\epsilon}$ touches $\lps$ at $(x_0,t_0)$,
    \begin{equation*}
        w(x_0,t_0) \geq \varphi^{\epsilon}(x_0,t_0)+\delta >\lps(x_0,t_0),
    \end{equation*}
    which is a contradiction. Thus, the lemma is proved.  
\end{proof}

\section{Barriers}\label{sec:barriers}
For a continuous and bounded function $g: \Rnp \setminus \Omega \to \R$, with a limit at infinity, we consider the boundary value problem 
\begin{align*} 
\begin{cases}
 \frac{\partial u}{\partial t} = \mathcal{L}_{\infty}u &  \textup{ in }  \Omega, \\
 u = g &  \textup{ on }  \RnpO.
\end{cases}
\end{align*}
In this section, we are going to show the existence of solutions to the boundary value problem above using barriers. In Section \ref{sec:Perronbarrier} we define regular boundary points and show the classic result that if there exists a barrier for each boundary point and $0<\alpha<1$, then there exists a unique viscosity solution, $u$, to \eqref{eq:Luf} such that $u=g$ on $\RnpO$. In Section \ref{sec:barrierscyl} we construct barriers for the boundary points of the cylinder, defined in \eqref{eq:cyl}. 

\subsection{Perron's solutions and regular boundary points}\label{sec:Perronbarrier}

\begin{definition}\label{def:regular}
 Let $g: \Rnp \setminus \Omega \to \R$ be any continuous and bounded function with a limit at infinity. A boundary point $\xi_0 \in \partial \Omega$ is called regular if 
\begin{equation}\label{eq:regular}
    \lim_{\xi \in \Omega: \xi \to \xi_0} \ups(\xi)=g(\xi_0)= \lim_{\xi \in \Omega: \xi \to \xi_0} \lps(\xi).
\end{equation} 
\end{definition}

We note that since $\lps(\xi_0) =g(\xi_0)$ for all $\xi_0 \in \RnpO$, and $g$ is continuous, it immediately follows from Definition \ref{def:regular} that $\lps$ is continuous at $\xi_0$. Therefore, $(\lps)^*(\xi_0) = \lim_{r \to 0} \sup_{\xi \in B_r(\xi_0)} \lps(\xi) = g(\xi_0)$. The same argument also applies to $\ups$, and thus,
\begin{equation}\label{eq:regular2}
    \lim_{\xi \to \xi_0} (\lps)^*(\xi_0) = \lim_{\xi \to \xi_0} \lps(\xi_0) =g(\xi_0)= \lim_{\xi \to \xi_0} \ups(\xi_0)  = \lim_{\xi \to \xi_0} (\ups)_*(\xi_0).
\end{equation}
 Next, we state the definition of barriers.
\begin{definition}\label{def:barrier}
   We let $\xi_0 \in \partial \Omega $. A function $\psi$ is a barrier for $\xi_0$ if 
   \begin{enumerate}
       \item $\psi \geq 0$, $\psi \in C(\Omega)$, and $\psi$ is a viscosity supersolution to \eqref{eq:Luf} ,
       \item For all $\delta >0$, $\inf \psi(\xi) >0$ for $ \xi \in \Rnp \setminus B_{\delta}(\xi_0)$,
       \item $ \psi(\xi_0)=0$.
   \end{enumerate}
\end{definition}
In the next proposition, we prove the classic result that a boundary point $\xi_0 \in \partial \Omega$ is regular if there exists a barrier at $\xi_0$.
\begin{prop}\label{prop:regular}
    Let $0<\alpha<1$ and $\xi_0 \in \partial \Omega $ be a boundary point. If there exists a barrier at $\xi_0$, then $\xi_0$ is a regular boundary point. 
\end{prop}
\begin{proof} 
 We choose $\epsilon >0$, $\xi_0 \in  \partial \Omega $ and note that  
  \begin{enumerate}
      \item  there exists a $\delta_1 >0$ such that $|g(\xi)-g(\xi_0)| < \epsilon$ for $\xi \in B_{\delta_1}(\xi_0) \cap (\RnpO)$,
      \item let $0<\delta_2 < \delta_1$. We can choose a constant $M$ large enough so that 
      $$
      M \inf_{\xi \in \Rnp \setminus B_{\delta_2}(\xi_0)}\psi(\xi) > 2 \sup |g|.
      $$
  \end{enumerate}
  The first point follows since $g$ is continuous, and the second point follows from the second condition in Definition \ref{def:barrier}. We define the function 
  \begin{equation*}
      w^-(\xi)= g(\xi_0)-\epsilon-M \psi(\xi), 
  \end{equation*}
  and we are going to show that $w^-\in \lcg$. It is clearly a continuous viscosity subsolution and points 1) and 2) implies that $w^- \leq g$ on $\RnpO$.  We consider $\tilde{\xi}\in \partial \Omega$. From the definition of $w^-$ it follows that $w^-(\xi) \leq g(\xi_0)-\epsilon$ and we note that this implies that
\begin{equation}\label{eq:limsup}
    \limsup_{\xi \to \tilde{\xi}} w^-(\xi) \leq g(\xi_0)-\epsilon.
\end{equation}
When $\tilde{\xi} \in B_{\delta_2}(\xi_0) \cap \partial \Omega$ it follows from \eqref{eq:limsup} and 1) that $ \limsup_{\xi \to \tilde{\xi}} w^-(\xi) \leq g(\tilde{\xi})$.  We next consider $\tilde{\xi } \in (\Rnp \setminus B_{\delta_2}(\xi_0)) \cap \partial \Omega$. Due to point 2) above,
  \begin{equation}\label{eq:wsup}
      w^-(\xi) \leq -\sup |g|-\epsilon,
  \end{equation}
when $\xi \in \Rnp \setminus B_{\delta_2}(\xi_0)$ and thus $ \limsup_{\xi \to \tilde{\xi}} w^-(\xi) \leq g(\tilde{\xi})$. Thus, $w^- \in \lcg$ so $\lps \geq w^-$. Furthermore, since $ \psi(\xi_0) = 0$, there exists some small $r>0$ such that $w^- \geq g(\xi_0)-2 \epsilon$ in $B_r(\xi_0) \cap \overline{\Omega}$. Thus, 
    \begin{equation}\label{eq:limlps}
       \liminf_{\xi \to \xi_0} \lps(\xi) \geq  g(\xi_0)-2\epsilon. 
    \end{equation}
We next define the function $w^+(\xi)= g(\xi_0)+\epsilon+M \psi(\xi)$
    and due to a similar argument as above we see that $w^+ \in \ug$. It follows that
    \begin{equation}\label{eq:limups}
        \limsup_{\xi \to \xi_0} \ups(\xi) \leq g(\xi_0)+2\epsilon.
    \end{equation}
    By combining \eqref{eq:limlps} and \eqref{eq:limups} and using that $\lps \leq \ups$ the proposition is proved.
\end{proof}

Next, we are going to prove the main theorem of this paper. Before that, we introduce the set of boundary points $\partial \tilde{\Omega}$ that contains all $(x_0,t_0) \in \partial \Omega$ with the possible exceptions of points for which $t_0=T.$ The reason that we do not necessarily need the ``top'' points is due to the way that the comparison principle is formulated.
\begin{thm}\label{thm:existence}
    Let $0<\alpha<1$ and let $g: \Rnp \setminus \Omega \to \R$ be a continuous and bounded function with a limit at infinity. If there exists a barrier for each boundary point $\xi_0 \in \partial \tilde{\Omega}$, then $\lps = \ups :=u$, where $u$ is a unique viscosity solution to \eqref{eq:Luf} such that $u=g$ on $\RnpO$ and $\lim_{\xi \to \xi_0}u(\xi)=g(\xi_0)$ for each $\xi_0 \in \partial \tilde{\Omega}$.
\end{thm}
\begin{proof}
   It follows from Lemma \ref{lem:supsubissub} and Lemma \ref{lem:Perronbdy} that $(\lps)^*$ is a viscosity subsolution to \eqref{eq:Luf} such that $(\lps)^*=g$ on $D_{\Omega}\setminus \Omega$. Furthermore, from Proposition \ref{prop:regular} we know that all boundary points $\xi_0 \in \partial \tilde{\Omega}$ are regular, and from \eqref{eq:regular2} we recall that this implies that $\lim_{\xi \to \xi_0}(\lps)^*(\xi) =g(\xi_0)$ for all $\xi_0 \in \partial \tilde{\Omega}$. In addition,  by Lemma \ref{lem:subissup} $\lps$ is a viscosity supersolution to \eqref{eq:Luf} such that $\lps = g$ on $D_{\Omega}\setminus \Omega$. From \eqref{eq:regular2} we know that $\lim_{\xi \to \xi_0}\lps(\xi)=g(\xi_0)$ for all $\xi_0 \in \partial \tilde{\Omega}$. By the comparison principle, it follows that $(\lps)^* \leq \lps$ in $\Omega$, i.e., $(\lps)^* = \lps$ in $\Omega$. Thus, $\lps$ is a viscosity solution to \eqref{eq:Luf} such that $\lps=g$ on $\RnpO$ and $\lim_{\xi \to \xi_0}\lps(\xi)=g(\xi_0)$ for $\xi_0 \in \partial \tilde{\Omega}$. We can apply a similar argument to see that $\ups$ is also a viscosity solution and that $\ups=g$ on $\RnpO$ and $\lim_{\xi \to \xi_0}\ups(\xi)=g(\xi_0)$ for $\xi_0 \in \partial \tilde{\Omega}$. It follows from the comparison principle that they coincide.
\end{proof}
\subsection{Semibarriers}\label{subsec:semibarriers}
We are also going to define so called semibarriers. The difference between a semibarrier and a barrier is that a semibarrier is not necessarily continuous.
\begin{definition}\label{def:semibarriers}
We let $\xi_0 \in \partial \Omega$. A function $\psi$ is a semibarrier for $\xi_0$ if 
\begin{enumerate}
       \item $\psi \geq 0$, and $\psi$ is a viscosity supersolution to \eqref{eq:Luf} ,
       \item For all $\delta >0$, $\inf \psi(\xi) >0$ for $ \xi \in \Rnp \setminus B_{\delta}(\xi_0)$,
       \item $\lim_{\xi \to \xi_0} \psi(\xi)=0$.
 \end{enumerate}
\end{definition}
\begin{lem}\label{lem:cones2}
Let $0<\alpha \leq 1$. If $\xi_0=(x_0,t_0) \in \partial \Omega$ is regular, then there exists a semibarrier for $\xi_0$. 
\end{lem}
\begin{proof}
    We define the function
\begin{equation}
    \varphi(x,t)= \min(|x-x_0|^2, R^2) + k_0(t-t_0)^2,
\end{equation}
where $k_0 >0$ and $ R \in \R$ are constants to be decided. At first we are going to show that $\varphi(x,t)$ is a subsolution in $\Omega$ and we therefore consider $(x,t) \in \Omega$. From Lemma \ref{lem:mincone} we recall that 
    \begin{align*}
     \Linf \min(|x-x_0|^2, R^2)= 
                      \left(R+|x-x_0| \right) \left(R-|x-x_0| \right)^{1-\alpha} - |x-x_0|^{2-\alpha} h(\alpha)
     \end{align*} 
     at each $x \in B_R(x_0)$,
     where 
     $$
     h(\alpha) = \left(1+\frac{\alpha}{2-\alpha} \right)\left(1-\frac{\alpha}{2-\alpha} \right)^{1-\alpha},
     $$
     and $1<h(\alpha)\leq2$ for $0<\alpha \leq1$. We choose 
     \begin{align*}
         R \geq \textup{max}\left(2\textup{diam}(\Omega_E)^{2-\alpha}+1, \textup{diam}(\Omega_E) +1\right),
     \end{align*}
     where  $\On$ is defined in \eqref{eq:projx}. With this choice of $R$ it follows that
     \begin{align*}
        \Linf \varphi(x,t) &\geq \left(R+|x-x_0| \right)\left(R-|x-x_0| \right)^{1-\alpha}-2|x-x_0|^{2-\alpha} \\
        & \geq R(R-\textup{diam}(\Omega_E))^{1-\alpha}-2\textup{diam}(\Omega_E)^{2-\alpha} \\ 
        & \geq R - 2\textup{diam}(\Omega_E)^{2-\alpha}  \\
        & \geq 1.
     \end{align*}     
       We next choose $k_0$ so small that $2k_0 (T-T_0) < 1 $, to conclude that $\varphi$ is a subsolution. We note that since $\varphi$ is a continuous subsolution, $\varphi$ belongs to $\mathcal{L}_{\varphi}$. Next, we are going to show that the function $\underline{H}_{\varphi}$ is a semibarrier at $\xi_0$. First, $\underline{H}_{\varphi}$ is a supersolution according to Lemma \ref{lem:subissup}. In addition, since $\underline{H}_{\varphi} \geq \varphi$ by definition, it is clear that the second condition in Definition \ref{def:semibarriers} is satisfied.
        Finally, since $\xi_0$ is a regular point it follows that
       $\lim_{\xi \to \xi_0} \underline{H}_{\varphi}(\xi)=\varphi(\xi_0) =0$, and thus the third condition is also satisfied. Thus, $\underline{H}_{\varphi}$ is a semibarrier for $\xi_0$.
\end{proof}

\subsection{Barriers for the cylinder}\label{sec:barrierscyl}
In this section, we let $E \subset \Rn$ and let $\cyl = E \times (T_0,T)$, for some $T_0, T \in \R$ denote the corresponding cylinder. We are considering the boundary points
\begin{equation*}
    \partial \tilde{Q}_T = \overline{E} \times \{T_0\} \cup \partial E \times (T_0, T]
\end{equation*}
Note that we do not consider the boundary points $E \times \{T\}$.
\begin{lem}\label{lem:barrierscyl}
    Let $0 < \alpha \leq 1$. There exist barriers for all $(x_0,t_0) \in \partial \tilde{Q}_T$.
\end{lem}
\begin{proof}
    We first consider the boundary points $(x_0,t_0) \in \overline{E} \times \{T_0 \}$. For these points, we define the function 
   \begin{align}\label{eq:barrierbottomcyl}
      \psi(x,t)= \varphi(x)+ k_{\alpha}(t-T_0), 
    \end{align}
 where
$$
\varphi(x)=  \min(|x-x_0|^2, R^2),
$$
and $k_{\alpha}>0, R \in \R$ are constants that will be defined later. From Lemma \ref{lem:mincone} we recall that 
    \begin{align*}
     \Linf \min(|x-x_0|^2, R^2)= 
                      \left(R+|x-x_0| \right) \left(R-|x-x_0| \right)^{1-\alpha} - |x-x_0|^{2-\alpha} h(\alpha)
     \end{align*} 
     at each $x \in B_R(x_0,t_0)$ and  $1<h(\alpha)\leq2$ for $0<\alpha \leq 1$. We choose $R> \textup{diam}(E)+1 $ and calculate
     \begin{align*}
         \Linf \min(|x-x_0|^2, R^2) &= 
                      \left(R+|x-x_0| \right) \left(R-|x-x_0| \right)^{1-\alpha} - |x-x_0|^{2-\alpha} h(\alpha)\\
                      &\leq \left(R + \textup{diam}(E) \right) R^{1-\alpha}:=c_{\alpha}.
     \end{align*}
     We choose $k_{\alpha} \geq c_{\alpha}$ so that
     \begin{align*}
         \frac{\partial \psi}{\partial t}(x,t)-\Linf \psi(x,t) \geq k_{\alpha}-c_{\alpha} \geq 0,
     \end{align*}
 which proves that $\psi$ is a continuously differentiable supersolution, In addition, it satisfies all the conditions of Definition \ref{def:barrier} and is therefore a barrier.
    
    Next we consider $(x_0,t_0) \in\partial E \times (T_0, T]$. We let $\beta < \alpha$ and define the function 
    \begin{equation}\label{eq:x0notinOx}
           \psi^{\beta}(x,t) = \varphi^{\beta}(x)+ (t-t_0)^2,
    \end{equation}
    where $\varphi^{\beta}(x)= \min(|x-x_0|^{\beta}, R^{\beta})$, and $k_{\beta}>0$ is a constant to be chosen. Clearly 
    $$
    \Lneg \varphi^{\beta}(x) \leq -|x-x_0|^{\beta-\alpha}.
    $$ 
    We next study $\Lpos \varphi^{\beta}(x)$, and note that in this case $x\neq x_0$ for $(x,t) \in \cyl$. By employing the same technique as in Lemma \ref{lem:mincone} we see that
    \begin{align*}
        \Lpos \varphi^{\beta}(x) &= \sup_{y \in \Rn, y\neq x_0} \frac{\min(|y-x_0|^{\beta}, R^{\beta})-|x-x_0|^{\beta}}{|y-x|^{\alpha}} \\
        &= \sup_{1<r\leq R/|x-x_0|}|x-x_0|^{\beta-\alpha}\frac{r^{\beta}-1}{(r-1)^{\alpha}}.
    \end{align*}
We are now in the exact same situation as in the proof of Lemma 8.3 in \cite{CLM2012} where it is shown that
\begin{equation*}
    \Linf \varphi^{\beta}(x) \leq -C(\alpha,\beta) |x-x_0|^{\beta-\alpha},
\end{equation*}
for a constant $C(\alpha, \beta)>0.$ Thus,
\begin{equation*}
    \frac{\partial \psi^{\beta}}{\partial t}(x,t)-\Linf \psi^{\beta}(x,t) \geq 2k_{\beta}(t-t_0) +C(\alpha, \beta) |x-x_0|^{\beta-\alpha} \geq 0.
\end{equation*}
 Thus, $\psi^{\beta}$ is a continuous viscosity supersolution in $\cyl$ and satisfies all the criteria for being a barrier.
\end{proof}

We obtain the following existence result. The proof is the same as the proof of Theorem \ref{thm:existence}. 
\begin{cor}
    Let $0<\alpha<1$, $\Omega = \cyl$, and $g: D_{\cyl} \setminus \cyl \to \R$ be a continuous and bounded function with a limit at infinity. Then there exists a unique viscosity solution, $u:= \lps =\ups$, to \eqref{eq:Luf} such that $u=g$ on $D_{\cyl} \setminus \cyl $ and $\lim_{\xi \to \xi_0} u(\xi) = g(\xi_0)$ for each $\xi_0 \in  \partial \tilde{Q}_T$.
\end{cor}

\subsubsection{Barriers for the more general case}

\begin{lem}\label{lem:barriersgen}
    Let $0<\alpha \leq 1$. There exist barriers for the boundary points $\xi_0=(x_0,t_0) \in \partial \Omega$ such that
        \begin{enumerate}
        \item $(x_0, t_0) \in  \overline{E} \times \{T_0\}$,
        \item $(x_0, t_0) \in \partial E \times (T_0, T]$,
       % \item $(x_0, t_0) $ is such that there exists a ball $B_r(\xi_0)$ such that $B^-_r(\xi_0) \cap \Omega = \emptyset$, where $B^-_r(\xi_0)$ is defined in \eqref{eq:halfball}.
    \end{enumerate}
\end{lem}
\begin{proof}
 For the first case we can use the function $\psi^{\alpha}$ in \eqref{eq:barrierbottomcyl} and for the second case we can use the function $\psi^{\beta}$ in \eqref{eq:x0notinOx}.

\begin{comment}

For 3) we only need to prove that there exists a barrier $\varphi$ for $\xi_0$ in $B_r(\xi_0)$ and then use the function $\psi$ in \eqref{eq:semibarrier} for $D=B_r(\xi_0)$. For this we can use the function 
    \begin{align*}
      \varphi(x,t)=          
      \begin{cases}
           f(x)+ t-t_0   & \textup{if  } t \geq t_0,\\
              f(x)+ t_0-t     & \textup{if  } t<t_0.
        \end{cases}
    \end{align*}
where
$$
f(x)=  \min(|x-x_0|^{\alpha}, R^{\alpha}),
$$
where $R$ is so large that $B_{2r}(x_0) \subset \subset B_R(x_0)$. Using the same argument as when proving that the bottom boundary points of the cylinder is regular it is clear that $\varphi$ is a continuous supersolution for $(x,t)\in \Omega\cap B_r(\xi_0)$ and it clearly satisfies all the conditions for being a barrier.
\end{comment}
\end{proof}

\appendix

\section{Appendix}
\begin{lem}\label{lem:simple}
             Let $0<\alpha \leq 1$ and 
     $$
     h(\alpha) = \left(1+\frac{\alpha}{2-\alpha} \right)\left(1-\frac{\alpha}{2-\alpha} \right)^{1-\alpha}.
     $$
     Then $1<h(\alpha)\leq 2$.
     \end{lem}
     \begin{proof}
             We note that $h(0)=1$ and $h(1)=2$. We next write $h(\alpha)$ as
    \begin{equation*}
       h(\alpha)= \left(1+f(\alpha)\right) \exp((1-\alpha)\ln(1-f(\alpha))),
    \end{equation*}
    where $f(\alpha)= \frac{\alpha}{2-\alpha}$. Then 
    \begin{align*}
        h'(\alpha) &= \exp((1-\alpha)\ln(1-f(\alpha))) \left(f'(\alpha)-\left(1+f(\alpha)\right)\left(\ln(1-f(\alpha)) +\frac{1-\alpha}{1-f(\alpha)}f'(\alpha) \right) \right)\\
        &= \exp((1-\alpha)\ln(1-f(\alpha))) \left(\frac{f'(\alpha)}{1-f(\alpha)}\underbrace{(\alpha +f(\alpha)(\alpha-2))}_{=0} -(1+f(\alpha))\ln(1-f(\alpha))\right)\\
        &= - \exp((1-\alpha)\ln(1-f(\alpha))) (1+f(\alpha))\ln(1-f(\alpha)).
    \end{align*}
    Since $0<1-f(\alpha)<1$ for $0<\alpha<1$ it follows that $h'(\alpha) >0$ for $0<\alpha<1$ which concludes the proof of the lemma.
     \end{proof}
 The following lemma is a standard result and can be found in Theorem 4.2 in \cite{K2004B} or Lemma 4.15 in \cite{FR2022}.
\begin{lem}\label{lem:strictmax}
    Let $\mathcal{A}$ be a family of upper semicontinuous functions. We let 
    \begin{equation*}
        u(x,t):= \sup_{v \in \mathcal{A}} v(x,t),
    \end{equation*}
and define $u^*$ as in \eqref{eq:u_upp}. We consider a point $(x_0, t_0) \in \Omega.$ Let $\Ns $ be a neighborhood of $(x_0,t_0)$ such that $(x_0,t_0)\in \Ns \subset \Omega$. Furthermore, let $\varphi \in C^1(\Ns)$ be a function that touches $u^*$ strictly from above at $(x_0,t_0)$. Then there exists a neighborhood $Q_r= B_r(x_0,t_0) \times (t_0-r,t_0+r)$ of $(x_0,t_0)$ such that $Q_r \subset \subset \Ns$, and a sequence of points $(x_k,t_k) \in Q_r$, such that $(x_k,t_k)\to (x_0, t_0)$, and functions $u_k \in \mathcal{A}$ such that 
    \begin{equation}\label{eq:supoverN}
         \sup_{Q_r} u_k-\varphi = u_k(x_k,t_k)-\varphi(x_k,t_k),
    \end{equation}
     and $u_k(x_k,t_k) \to u^*(x_0, t_0)$.
\end{lem}

\begin{proof}     
     We take any $Q_r \subset \subset \Ns$  and choose $N_{\rho} \subset \subset Q_r$ such that $(x_0,t_0) \in N_{\rho}$. We next define
    \begin{equation}\label{eq:supoutside}
        \sup_{Q_r \setminus N_{\rho}} u^*-\varphi = -s.
    \end{equation}
    The supremum exists since $u^*$ is upper semicontinuous and $\varphi$ is continuous. Furthermore, $s>0$ since the touching is strict. Because otherwise there exists a sequence of points $\{\tilde{z}_j\} \in Q_r \setminus N_{\rho}$ such that $u^*(\tilde{z}_j)-\varphi(\tilde{z}_j) \to 0$. However, $u^*-\varphi$ is upper semi continuous and this implies that, for a subsequence, $\tilde{z} = \lim_{j \to \infty }\tilde{z}_j$, we have that
    $$
    0 = \limsup_{\tilde{z}_j \to \tilde{z}} u^*(\tilde{z}_j)-\varphi(\tilde{z}_j) \leq u^*(\tilde{z})-\varphi(\tilde{z}).
    $$
    which is a contradiction since $ u^* < \varphi$ in $\overline{Q_r} \setminus N_{\rho}$. Due to the upper semicontinuity of $u$, we can choose points $z_k \to (x_0,t_0)$ such that $u(z_k) \geq u^*(x_0,t_0)-\frac{1}{k}$. From the definition of $u$ there exists a $u_k \in \mathcal{A}$ such that $u(z_k) \leq u_k(z_k)+ \frac{1}{k}$. Thus,
    \begin{equation}\label{eq:ukgreat}
        u_k(z_k)\geq u^*(x_0,t_0) -\frac{2}{k}.
    \end{equation}
    We choose $k$ so large that $3/k < s$ and
    \begin{equation}\label{eq:varphigreat}
        |\varphi(z_k)-\varphi(x_0,t_0)| < 1/k.
    \end{equation}
    By combining \eqref{eq:ukgreat} and \eqref{eq:varphigreat} and using that $u^*(x_0,t_0)=\varphi(x_0,t_0)$, we see that
    \begin{equation}\label{eq:ukvarphi}
        u_k(z_k)-\varphi(z_k) > - \frac{3}{k} >-s.
    \end{equation}
    Thus, if we combine \eqref{eq:supoutside} and \eqref{eq:ukvarphi} it follows that
    \begin{equation}
         \sup_{Q_r \setminus N_{\rho}} u_k-\varphi < u_k(z_k)-\varphi(z_k).
    \end{equation}
    Thus, for large values of $k$ there exists a $y_k \in N_{\rho}$ such that \eqref{eq:supoverN} holds.
   For a subsequence we have that
   $y = \lim_{k \to \infty} y_k$ for some $y \in \overline{N_{\rho}}$. Thus,
   $$
  0\leq u_k(z_k)-\varphi(z_k)+\frac{3}{k} \leq  u_k(y_k)-\varphi(y_k)+\frac{3}{k}  \leq u^*(y_k)-\varphi(y_k)+\frac{3}{k},
   $$
   where we have used \eqref{eq:ukvarphi} in the first inequality. We let $k\to \infty$ and use that $u^*$ is upper semicontinuous to see that $u^*(y)-\varphi(y) \geq0$. Thus, $(x_0,t_0)=y$, and $u_k(x_k,t_k) \to u^*(x_0,t_0)$. 
\end{proof}

%\begin{proof}[Proof of first part of Proposition \ref{prop:regular}]
 
%\end{proof}

\bibliographystyle{abbrv}
\bibliography{bibliography}

\end{document}